\documentclass[a4paper,10pt]{amsart}
\usepackage[utf8]{inputenc}
\usepackage{amsmath,amssymb}
\usepackage{tikz-cd}
\usetikzlibrary{cd}

\newtheorem{theorem}{Theorem}[section]
\newtheorem{lemma}[theorem]{Lemma}
\newtheorem{exmp}[theorem]{Example}
\newtheorem{definition}{Definition}
\newtheorem{remark}{Remark}
\newtheorem{corollary}{Corollary}
\newtheorem{question}{Question}

\title{Effective computation of degree bounded minimal models of GCDA's}
\author{Victor Manero\\
Miguel \'Angel Marco Buzun\'ariz
}
\thanks{First author is partially supported by MTM2017-85649-P (AEI/Feder, UE) and ``\'Algebra y Geometr\'ia'' (Gobierno de Arag\'on/FEDER).
Second author is partially supported by MTM2016-76868-C2-2-P and Grupo ``Investigaci\'on en Educaci\'on Matem\'atica'' of Gobierno de Arag\'on/Fondo Social Europeo.}

\begin{document}

\begin{abstract}
Given a finitely presented Graded Commutative Differential Algebra (GCDA),
we present a method to compute its minimal model, together with a map that is a quasi-isomorphism up to a given degree. The method works by adding generators one by one. We also provide a specific implementation of the
method.
We also provide two criteria for $i$-formality, one necessary and one sufficient.
\end{abstract}

\maketitle

\section{Introduction}

\begin{definition}
\label{dfn:cdga}
 A \emph{Graded Commutative Differential Algebra} (or GCDA) is a graded algebra $\displaystyle{A=\bigoplus_{i=0}^\infty}A_i$, together with a
 linear map $d_A:A\rightarrow A$ that satisfies the following conditions:
 \begin{itemize}
  \item $A_iA_j \subseteq A_{i+j}$
  \item $d_A^2= 0$
  \item $d_A(A_i)\subseteq A_{i+1} \hspace{20pt} \forall i\in\mathbb{N}$
  \item $ab=(-1)^{ij} ba \hspace{20pt} \forall a\in A_i, b\in A_j$
  \item $d_A(ab) = d_A(a)b+(-1)^iad_A(b) \hspace{20pt} \forall a\in A_i, b\in A$
 \end{itemize}
\end{definition}
From now on we will refer to a GCDA as differential algebra and when no confusion could be produced we will denote the pair $(A, d_A)$ simply by the name of their corresponding algebra $A$. To simplify more the notation, we will assume that the base field is $\mathbb{Q}$, although the results will hold for more general fields. We will also assume that $A_0=\mathbb{Q}$ what is called in the related literature to be connected.

We will consider finitely presented differential algebras, given by the following data:
\begin{itemize}
 \item A finite set of homogenous \emph{generators} $\{a_1,\ldots,a_n\}$
 \item For each generator $a_j$, a positive integer that will be its \emph{degree} which is usually denoted as $|a_j|$.
 \item For each generator of degree $i$, its differential, which is either zero or a homogeneous supercommutative polynomial in the generators, of degree $i+1$.
 \item A finite set of homogeneous relations $\{R_1,\ldots, R_m\}$, which are supercommutative polynomials in the generators.
\end{itemize}

Given this data, $A$ is the quotient of the ring of supercommutative polynomials in $a_1,\ldots,a_n$ by the twosided ideal $I$ generated by the relations $(R_1,\ldots,R_m)$. A Gr\"obner basis for $I$ allows us to have
a normal form (and hence, a unique representation) for the elements in $A$.
The sum, product and differential can be computed by using the properties of Definition~\ref{dfn:cdga}.


There exist very well known examples of differential algebras like the De Rham complex $(\Omega^*(N),d_N)$  of differential forms on a manifold $N$ with the differential $d_N$ or the De Rham cohomology algebra $(H^*(N),d=0)$ with null differential. Notice that, more in general, every differential algebra $(A,d_A)$  has associated another differential algebra which is exactly the algebra given by its cohomology with null differential, i.e. $\big(H^{\ast}(A),d=0\big)$.

\begin{definition}
\label{dfn:minimal}
A differential algebra $A$ is said to be \emph{$i$-minimal} (in the sense of Sullivan) if it is freely generated by a collection of elements $\{a_l\}$ with $l \in J$, of $V$, for a well-ordered index set $J$, such that $|a_l|\leq|a_s|\leq i$ if $l < s$ and the differential of a generator $a_s$ is expressed in terms of the preceding $a_l$ with $l < s$.
\end{definition}
Consider $A$ and $B$ differential algebras, a morphism of algebras
$$\phi:A \rightarrow B$$
is said to be a morphism of differential algebras if it preserves the degree and commutes with the differential. Those differential algebras are said to be \emph{i-quasi-isomorphic} if there exists a morphism of differential algebras $\phi: A \longrightarrow B$ such that $\phi^{\ast}: H^j(A) \longrightarrow H^j(B)$ is an isomorphism for every $j \leq i$ and  $\phi^{\ast}: H^{i+1}(A) \longrightarrow H^{i+1}(B)$ is a monomorphism.

\begin{definition}\label{def3}
A differential algebra $(M, d_M)$ is an  $i$-minimal model of the differential algebra $(A, d_A)$ if $(M, d_M)$ is $i$-minimal and there exists an $i$-quasi-isomorphism
\begin{equation*}
\phi: (M,d_M) \rightarrow (A,d_A).
 \end{equation*}

\end{definition}

The existence and uniqueness of the minimal model of a connected differential algebra is guaranteed by the following result due to Halperin.
\begin{theorem}\cite{Morgan, Sullivan}
Every differential algebra $A$ which is connected (i.e $H^0(A)=\mathbb{Q}$) has a unique $i$-minimal model $M$ up to isomorphism for every $i\geq 0$. 
\end{theorem}

\section{Effective computation of the minimal model}\label{sec:algorithm}
In this section we give a description of an algorithm to compute the $i$-minimal model of a given differential algebra. It follows the usual approach in the literature, but presented as an explicit algorithm.

We aim to construct a differential algebra $M$ that is an $i$-minimal model of $A$, together with
a morphism of differential algebras $\phi:M\rightarrow A$ being an $i$-quasi isomorphism. In order to describe it, we need a set of generators as in Definition~\ref{dfn:minimal}, and for each of them, its degree, differential and image by $\phi$. The method consists on adding these generators sequentially, in such a way that we get the needed conditions. We will use the following notation:

\begin{itemize}
 \item $x^k_j$ will denote a generator of $M$ of degree $k$, whose differential is zero.
 \item $y^k_j$ will denote a generator of $M$ of degree $k$, with nonzero differential.
\end{itemize}

Each time we add a new generator to $M$, we will use the following diagram
\[
\begin{array}{ccccc}
b^k_j & \leftarrow & y^k_j & \rightarrow & z^{k+1}_j
\end{array}
\]
to denote that we have added the generator $y^k_j$, with $d_M(y^k_j)=z^{k+1}_j$ and
$\phi(y^k_j)=b^k_j$.

\subsection{First step}
\label{step:x0}

We start by computing the smallest $k_0$ for which $H^{k_0}(A)$ is not trivial.
Take a basis $[{a}^{k_0}_0],\ldots, [{a}^{k_0}_{l_{k_0}}]$ of $H^{k_0}(A)$, and choose representatives $a^{k_0}_0\ldots,a^{k_0}_{l_{k_0}}$ in $A$. The first generators to add to $M$ are $x^{k_0}_0,\ldots x^{k_0}_{l_{k_0}}$. That is, we start with the diagram

\[
\begin{array}{ccccc}
    a^{k_0}_1 & \leftarrow & x^{k_0}_0 & \rightarrow & 0 \\
& & \vdots & & \\
a^{k_0}_{l_{k_0}} & \leftarrow & x^{k_0}_{l_{k_0}} & \rightarrow & 0
\end{array}
\]

At this point, we have that $\phi:M\rightarrow A$ induces an isomorphism $H^{j}(M)\rightarrow H^{j}(A)$ for all $j\leq n_0$.

\subsection{Increase degree}

Assume that we have already found generators of degree up to $k-1$, such that the map $\phi:M\rightarrow A$ induces isomorphisms $\phi^*_m:H^{m}(M)\rightarrow H^{m}(A)$ for all $m\leq k-1$. In this step, we will add new generators to get also an isomorphism $\phi^*_k:H^k(M)\rightarrow H^k(A)$, without changing the lower degree cohomologies.

This step has two phases. In the first phase, that might need to be run iteratively, we will add generators of the form $y^{k-1}_j$ until the obtained map in $\phi^*_k$ is injective.

Once we have an injective map, we will add generators of type $x^{k}_j$ (that is, with zero differential), which will respect the lower degree differentials, and the injectivity at degree $k$, until we get a surjective map.


\subsubsection{Adding generators of the form $y^{k-1}_j$}
\label{step:y}

Compute the $k$'th cohomology group $H^k(M)$ with the already defined
generators, and the induced map $\phi^*_k:H^k(M)\rightarrow H^k(A)$. If this
map $\phi^*_k$ is not injective, take $[{z}^k_0], \ldots,[{z}^k_{l^1_{k-1}}]$ a basis
of $Ker(\phi^*_k)$. And consider representatives $z^k_0,\ldots z^k_{l^1_{k-1}} \in M$.
Compute $c^k_j = \phi(z^k_j)\in A_k$. Since $[{z}^k_j]$ is in the kernel
of the cohomology map, $c^k_j$ must correspond to a trivial cohomology
class, that is, there must be $b^{k-1}_j\in A_{k-1}$ such that
$d_A(b^{k-1}_j)=c^k_j$. So we add the following generators:

\[
 \begin{array}{ccccc}
b^{k-1}_0 & \leftarrow & y^{k-1}_0 & \rightarrow & z^k_0 \\
& & \vdots \\
b^{k-1}_{l^1_{k-1}} & \leftarrow & y^{k-1}_{l^1_{k-1}} & \rightarrow & z^k_{l^1_{k-1}}
 \end{array}
\]

Note that, after adding these generators, new elements of  $Ker(\phi^*_k)$
could have been added. We can iterate this process if required adding new generators:
\[
 \begin{array}{ccccc}
b^{k-1}_0 & \leftarrow & y^{k-1}_0 & \rightarrow & z^k_0 \\
& & \vdots \\
b^{k-1}_{l^1_{k-1}} & \leftarrow & y^{k-1}_{l^1_{k-1}} & \rightarrow & z^k_{l^1_{k-1}} \\
b^{k-1}_{l^1_{k-1}+1} & \leftarrow & y^{k-1}_{l^1_{k-1}+1} & \rightarrow & z^k_{l^1_{k-1}+1} \\
& & \vdots \\
b^{k-1}_{l^1_{k-1}+l^2_{k-1}}& \leftarrow & y^{k-1}_{l^1_{k-1}+l^2_{k-1}} & \rightarrow & z^k_{l^1_{k-1}+l^2_{k-1}} \\
& & \vdots
 \end{array}
\]

until the map $\phi^*_k$ is injective.

\subsubsection{Adding generators of the form $x^d_i$}
\label{step:xd}

If this map $\phi^*_k$ is not surjective, get a basis $[{a}^k_1],\ldots,[{a}^k_{j}],[{a}^k_{j+1}],\ldots,[{a}^k_{l_k}]$, of $H^k(A)$ where $[{a}^k_{j+1}],\ldots,[{a}^k_{l_k}]$ is a basis of $Im(\phi^*_k)$. As before, take representatives ${a}^k_1,\ldots,{a}^k_{j}$, and construct the new generators

\[
 \begin{array}{ccccc}
a^{k}_1  & \leftarrow &  x^{k}_1 & \rightarrow & 0 \\
& & \vdots\\
a^{k}_{j} & \leftarrow & x^{k}_j & \rightarrow & 0
 \end{array}
\]

We repeat these two steps until $k=i$. Then we repeat step~\ref{step:y} one last time to get injectivity in $\phi^*_{i+1}$.

\begin{exmp}
    \label{exm:algorithm}
    Consider $A$ the algebra generated by six generators of degree $1$, $e_1,\ldots, e_6$ and one generator of degree $2$, $e_7$ with no extra relations, and the differential is given by:
 \begin{align*}
 d_A(e_1)&=-e_1\wedge e_6, \quad d_A(e_2)=-e_2\wedge e_6, \quad  d_A(e_3)=-e_3\wedge e_6, \\
 d_A(e_4)&=-e_5\wedge e_6, \qquad \text{ and } \qquad d_A(e_5)=d_A(e_6)=d_A(e_7)=0.
 \end{align*}

We can compute the first cohomology groups:
\begin{itemize}
\item $H^1(A)=\mathbb{Q}\langle e_5,e_6 \rangle$
\item $H^2(A)=\mathbb{Q}\langle e_4\wedge e_5, e_4\wedge e_6 , e_7 \rangle$
\item $H^3(A)=\mathbb{Q}\langle e_4\wedge e_5 \wedge e_6, e_5\wedge e_7, e_6\wedge e_7 \rangle$
\end{itemize}

So we can start the method at degree $1$:

\[
 \begin{array}{ccccc}
e_6  & \leftarrow & x^{1}_0 &\rightarrow & 0 \\
e_5 & \leftarrow & x^{1}_1 & \rightarrow & 0
 \end{array}
\]

At this point, $H^2(M)$ would be generated by the class of $x^1_0\wedge x^1_1$. But the image by the induced map is $-e_5\wedge e_6$, which is trivial in cohomology, because it is the differential of $e_4$. So we have to add a new generator in
order to make this cohomology class trivial:
\[
\begin{array}{ccccc}
e_4 & \leftarrow & y^{1}_0 & \rightarrow & x^1_0\wedge  x^1_1
\end{array}
\]

Now we get that $d_M(x^1_0\wedge y^1_0)=d_M(x^1_1\wedge y^1_1)=0$, so a basis of $H^2(M)$ is  formed by the classes of $(x^1_0\wedge y^1_0,x^1_1\wedge y^1_0)$. Their images by $\phi$ are precisely $e_4\wedge e_6$ and $e_4\wedge e_5$, which are two elements of the basis of $H^2(A)$. That is, we have already an injective map at the second cohomology level. Proceeding as described in step~\ref{step:xd}; that is, add a new generator that will be mapped to the remaining element of the basis of $H^2(A)$:

\[
\begin{array}{ccccc}
e_7 &  \leftarrow & x^{2}_0 & \rightarrow & 0
\end{array}
\]

Concerning degree $3$, a basis for the cohomology of the already computed $M$ is given by the classes of $x^1_0\wedge x^1_1\wedge y^1_0$, $x^1_0\wedge x^2_0$ and $x^1_1\wedge x^2_0$. Since these elements map precisely to the basis of $H^3(A)$, we already have an isomorphism in degree $3$. Up to degree 3 the minimal model of $A$ is
\begin{equation*}
M_3=\bigwedge(x^1_0,x^1_1, y^1_0) \otimes \bigwedge(x^2_0) \text{ and } d_M(y^1_0)=x^1_0\wedge x^1_1.
\end{equation*}

In fact, in this case we can check that the obtained result is a minimal model for $A$ in any degree.

We can summarize the whole process with the complete diagram:

\begin{equation}
    \label{diagram:ejemplo}
\begin{array}{ccccc}
e_6  & \leftarrow & x^{1}_0 &\rightarrow & 0 \\
e_5 & \leftarrow & x^{1}_1 & \rightarrow & 0 \\
e_4 & \leftarrow & y^{1}_0 & \rightarrow & x^1_0\wedge  x^1_1 \\
e_7 &  \leftarrow & x^{2}_0 & \rightarrow & 0
\end{array}
\end{equation}

\end{exmp}

\section{Proof of correctness}

\begin{lemma}
The algebra $M$ obtained after each step of the previous process is minimal.
\end{lemma}
\begin{proof}
It is free because we do not add any relation at any moment of the process. The generators are added in increasing order, and the differential of each generator is always either zero or expressed in terms of the previous generators.
\end{proof}

\begin{lemma}
Let $\phi: M \rightarrow A$ be the map obtained after step~\ref{step:x0}. The induced map $H^{k_0}(M) \rightarrow H^{k_0}(A)$ is an isomorphism.
\end{lemma}
\begin{proof}
Since each $x^{k_0}_j$ has zero differential, and $M^{k_0-1}=0$, they produce a basis of the cohomology group $H^{k_0}(M)$. It is clear that the map gives a bijection with a basis of $H^{k_0}(A)$.
\end{proof}

\begin{lemma}
Assume that before an iteration of the step~\ref{step:y} the maps
$\phi_m^*$ are bijective for all $m<k$, then they are also bijective after adding each generator $y^{k-1}_j$.
\end{lemma}
\begin{proof}
For the groups $H^m(M)$ with $m<k-1$, both the space of cocycles and coboundaries remain untouched.

The addition of $y^{k-1}_l$ produces a decomposition $M_{k-1}\bigoplus\mathbb{Q} y^{k-1}_l$, where $M_{k-1}$ is the space of homogeneous supercommutative polynomials of degree $k-1$ on the previous generators. This induces also a decomposition

$$Ker(d_M^{k-1})=\left(Ker(d_M^{k-1})\cap M_{k-1}\right)\bigoplus\left(Ker(d_M^{k-1})
\cap \mathbb{Q}y^{k-1}_l\right)$$
By construction, $d_M(y^{k-1}_l) \notin d_M(M_{k-1})$, so $\left(Ker(d_M^{k-1})
\cap \mathbb{Q}y^{k-1}_l\right)=0$. Hence at this degree also both cocycles and coboundaries remain untouched.
\end{proof}

\begin{lemma}
If step~\ref{step:y} terminates, the resulting map $\phi:M\rightarrow A$ is a $(k-1)-quasi-isomorphism$.
\end{lemma}

\begin{proof}
The step only terminates if the kernel of the induced map in degree $k$ is zero.
\end{proof}

\begin{lemma}
After step~\ref{step:xd}, the induced maps $H^k(M)\rightarrow H^k(A)$ are isomorphisms for $i\leq k$
\end{lemma}
\begin{proof}
Consider the map $H^k(M)\rightarrow H^k(A)$. By the previous lemmas, it was injective before step~\ref{step:xd}. The effect of adding the generators $x^k_1,\ldots, x^k_j$ is precisely to extend the vector space $H^k(M)$ with the needed generators to fill $H^k(A)$.
\end{proof}

Joining this results, we get that, if the method terminates at degree $i$, we have obtained an $i$-minimal model for $A$.

\section{Application to formality criteria}

In this section we will see how the previous method can be useful to determine the fomality of a differential algebra.

\begin{definition}
    A differential algebra $A$ is said to be \emph{$i$-formal} if its $i$-minimal model is $i$-quasi-isomorphic to its cohomology algebra $H^*(A)$. Analogously, $A$ is said to be \emph{formal} if it is $i$-formal for every $i\in\mathbb{Z}^+$.
\end{definition}

\begin{remark}
This definition is equivalent to the fact that the $i$-minimal model of $A$ is also an $i$-minimal model of $H^*(A)$.
\end{remark}

The notion of $i$-formality can be found in the literature under this same name. However, some authors call it $k$-stage formality (see \cite{macinic}). It should
be noted that there exist a different notion that also receives the name of
$s$-formality, introduced by Fern\'andez and Mu\~noz in \cite{fernandez}.

We will now see two criteria (one necessary and one sufficient) for $i$-formality that can be computed making use of the method described before.

\subsection{Necessary criterion: numerical invariants}

Let $A$ be a differential algebra. By using the previous method, we can compute an $i$-minimal model $M_A\rightarrow A$. We can also compute a
presentation of the cohomology algebra $H^*(A)$ up to degree $i+1$, and then
its $i$-minimal model $M_H\rightarrow H^*(A)$. By the previous remark, $A$ is $i$-formal if and only if $M_A$ and $M_H$ are isomorphic.

Determinining whether two presentations correspond to isomorphic  algebras or not is, in general, a hard problem. Therefore we will use some numerical invariants that are related to the construction process described in section~\ref{sec:algorithm}, and hence can be computed. In particular, these invariants coincide with the number of generators that are added in each step of the algorithm. Let us now see that these numbers are, in fact, invariants under isomorphism.

Let $M$ be a minimal algebra obtained by the method in section~\ref{sec:algorithm}. Consider the corresponding cochain complex

\[
\begin{tikzcd}
M^0 \arrow[r, "d_0"] & M^1 \arrow[r, "d_1"] &  M^2 \arrow[r, "d_2"] & \ldots
\end{tikzcd}
\]

We will define recursively two families of linear subspaces (denoted $V_j^i$ and $W_j^i$) and a family of subalgebras (denoted by $N_j^i$):

\begin{itemize}
    \item $V_0^0:=\{0\}$.
    \item $N_j^i$ is the subalgebra generated by all the vector spaces $V_k^l$ for $l<i$ or $l=i,k\leq j$.
    \item $N_\infty^i=\bigcup_{k=0}^\infty N_k^i$
    \item $W_0^i:=Ker(d_i)\cap N_\infty^{i-1}$ for $i>0$
    \item $V_0^i:=Ker(d_i)$.
    \item $W_{j+1}^i:=d_i^{-1}(N_{j}^i) \cap N_j^i$ for $i>0,j\geq 0$
    \item $V_{j+1}^i:=d_i^{-1}(N_{j}^i)$ for $i>0,j\geq 0$

\end{itemize}

It is clear that these subalgebras and vector spaces must be preserved by isomorphism of differential algebras, since their definition only involves the algebra structure and the differential.

Now consider the numbers $v_j^i:=dim(V_j^i/W_j^i)$. Let us now see that these numbers are related to the steps followed in the algorithm of section~\ref{sec:algorithm}.

Denote by $M^{k_0}_0$ the algebra obtained after step \ref{step:x0} of the algorithm, $M^i_j$ the algebra obtained after the $j$'th iteration of step~\ref{step:y} in degree $i$, and $M^i_0$ the algebra obtained after step~\ref{step:xd} in degree $i$.

The following diagram summarizes the inclusion $N^i_j\hookrightarrow M$:

\[
\begin{tikzcd}
    \cdots \arrow[r, "d_{i-2}"]  & \left(N^{i}_j\right)^{i-1} \arrow[r, "d_{i-1}"] \arrow[d, "\iota"] &  \left(N^i_j\right)^i\arrow[r, "d_i"]  \arrow[d, "\iota"] &  \left(N^i_j\right)^{i+1}\arrow[r, "d_{i+1}"] \arrow[d, "\iota"]& \ldots \\
\cdots \arrow[r, "d_{i-2}"] & \left(M\right)^{i-1} \arrow[r, "d_{i-1}"] &  \left(M\right)^i\arrow[r, "d_i"] &  \left(M\right)^{i+1}\arrow[r, "d_{i+1}"] & \ldots
\end{tikzcd}
\]

\begin{lemma}
The following properties hold
    \begin{itemize}
     \item $N^i_j=M^i_j$ for every $i,j$.
     \item $\iota^*:H^k\left(N^i_j\right)\rightarrow H^k\left(M\right)$ is surjective if $k\leq i$.
     \item $\left(N^i_j\right)^k=\left(M\right)^k$ if $k<i$.
     \item The numerical invariant $v^i_j$ coincides with the number of generators added in the algorithm in the step that corresponds to $M^i_j$.
    \end{itemize}

\end{lemma}

The proof can be done by induction over the steps of the algorithm.

%
%

%

\begin{corollary}
 If $A$ is formal, the numerical invariants $v_j^i$ of $M_A$ coincide with the ones of $M_H$.
\end{corollary}

\subsection{Sufficient criterion}

As before, let $A$ be a differential algebra, and $M_A$ the minimal model obtained by the previous method. Consider the corresponding cohomology algebra $H^*(A)\cong H^*(M_A)$. We can compute its minimal model $M_H$. Assume that the numerical invariants of $M_A$ and $M_H$ coincide.

Consider the diagram followed to compute $M_A$:

\[
 \begin{array}{ccccc}
a^{k_0}_0  & \leftarrow &  x^{k_0}_0 & \rightarrow &  0 \\
& & \vdots\\
a^{k_0}_{l_{k_0}}& \leftarrow &  x^{k_0}_{l_{k_0}} & \rightarrow &  0 \\
& & \vdots \\
b^{k-1}_0& \leftarrow &  y^{k-1}_0 & \rightarrow &  z^k_0 \\
& & \vdots \\
b^{k-1}_{l^1_{k-1}}& \leftarrow &  y^{k-1}_{l^1_{k-1}} & \rightarrow &  z^k_{l^1_{k-1}} \\
b^{k-1}_{l^1_{k-1}+1}& \leftarrow &  y^{k-1}_{l^1_{k-1}+1} & \rightarrow &  z^k_{l^1_{k-1}+1} \\
& & \vdots \\
b^{k-1}_{l^1_{k-1}+l^2_{k-1}}& \leftarrow &  y^{k-1}_{l^1_{k-1}+l^2_{k-1}} & \rightarrow &  z^k_{l^1_{k-1}+l^2_{k-1}} \\
& & \vdots \\
a^{k}_0  & \leftarrow &  x^{k}_0 & \rightarrow &  0 \\
& & \vdots \\
a^{k}_{l_k}& \leftarrow &  x^{k}_{l_k} & \rightarrow &  0 \\
& & \vdots
 \end{array}
\]

We define the morphism of algebras

$$
\begin{array}{cccc}
\psi: & M_A & \rightarrow & H^*(M_A) \\
& x^i_j & \rightarrow & [x^i_j] \\
& y^i_j & \rightarrow & 0
\end{array}
$$


\begin{definition}
    We will say that $M_A$ satisfies the \emph{$\psi$-condition} if $\psi(z^k_j)=0$ for every $z^k_j$.

    Analogously, we say that $M_A$ satisfies the \emph{$\psi$-condition up to degree $i$} if $\psi(z^k_j)=0$ for every $z^k_j$ with $k\leq i+1$.
\end{definition}

Note that the previous definition is equivalent to asking that $\psi$ is a morphism of \emph{differential} algebras $\psi:M_A\to H^*(M_A)$.

\begin{lemma}
 If $M$ satisfies the $\psi$-condition up to degree $i$, and the numerical invariants of $M_A$ coincide, up to degree $i$, with the ones of $M_H$, then $\psi$ is a $i$-quasi-isomorphism.
\end{lemma}

\begin{proof}
We will see it by proving that the minimal model of $H^*(M)$ can be computed using the diagram

\begin{equation}
\label{diagram:psicondition}
 \begin{array}{ccccc}
     \left[ x^{k_0}_0 \right]   & \leftarrow & x^{k_0}_0 & \rightarrow & 0 \\
& & \vdots & & \\
\left[ x^{k_0}_{l_{k_0}} \right] & \leftarrow & x^{k_0}_{l_{k_0}} & \rightarrow & 0 \\
& & \vdots \\
0 & \leftarrow & y^{k-1}_0 & \rightarrow  & z^k_0 \\
& & \vdots \\
0 & \leftarrow & y^{k-1}_{l^1_{k-1}} & \rightarrow & z^k_{l^1_{k-1}} \\
0 & \leftarrow & y^{k-1}_{l^1_{k-1}+1} & \rightarrow & z^k_{l^1_{k-1}+1} \\
& & \vdots & & \\
0 & \leftarrow & y^{k-1}_{l^1_{k-1}+l^2_{k-1}} & \rightarrow & z^k_{l^1_{k-1}+l^2_{k-1}} \\
& & \vdots \\
\left[ x^{k}_0 \right] & \leftarrow & x^{k}_0 & \rightarrow & 0 \\
& & \vdots \\
\left[ x^{k}_{l_k} \right] & \leftarrow & x^{k}_{l_k} & \rightarrow & 0 \\
& & \vdots
 \end{array}
\end{equation}

The proof will be done by induction on the steps of the algorithm in~\ref{sec:algorithm}.

Note that the left half of the diagram is in fact the map $\psi$. That is, in this case, $\psi$ will play the role of $\phi$.

In step~\ref{step:x0} we have to choose a basis of the first nontrivial cohomology group of $H^*(M)$. Since $H^*(M)$ is itself the cohomology algebra of $M$, its cohomology is isomorphic to itself. So we can choose
$[x^{k_0}_0],\ldots,[x^{k_0}_{l_{k_0}}]$ as the basis of its first nonzero cohomology group. Hence, we can start the construction of the minimal model of $H^*(M)$ with the diagram

\[
 \begin{array}{ccccc}
     \left[ x^{k_0}_0 \right]   & \leftarrow & x^{k_0}_0 & \rightarrow & 0 \\
& & \vdots & & \\
\left[ x^{k_0}_{l_{k_0}} \right] & \leftarrow & x^{k_0}_{l_{k_0}} & \rightarrow & 0
\end{array}
\]

Now, for each iteration of step~\ref{step:y}. By induction hypothesis, assume that the diagram used up to this step coincides with diagram~\eqref{diagram:psicondition}. To proceed with the iteration of the step, we need to choose a basis of $Ker(\psi^*_k)$. Since the numerical invariants coincide, the dimension of this basis has to coincide with the number of generators added in this step of the construction of $M$. The cohomology classes of $z^k_0,\ldots z^k_{l^1_{k-1}}$ do live in $Ker(\psi^*_k)$ because of the $\psi$-condition. And they are linearly independent on the cohomology of $M_A$ because of the induction hypothesis (the cohomology of $M_A$ up to this point has to coincide with the one of $M_H$). So the same $z^k_0,\ldots z^k_{l^1_{k-1}}$ that were used for constructing $M_A$ can be chosen to construct $M_H$.

For the left part of the diagram, we have to choose preimages by the differential of $\psi(z^k_0),\ldots \psi(z^k_{l^1_{k-1}})$. Since all these elements are zero, we can choose zero as its preimages, so we can extend the diagram with

\[
\begin{array}{ccccc}
    0 & \leftarrow & y^{k-1}_0 & \rightarrow & z^k_{0} \\
    & & \vdots \\
    0 & \leftarrow & y^{k-1}_{l^1_{k-1}} & \rightarrow & z^k_{l^1_{k-1}}
\end{array}
\]

For the steps~\ref{step:xd}, again the condition on the numerical invariants tells us that we have to add the same number of generators to the diagram. As before, the classes $[x^k_1]\ldots,[x^k_j]$ are linearly independent, in $H^k(M_A)$, and are not in the image of the previously added generators by $\psi$, so they are a suitable choice for this step.

\end{proof}

Thanks to the previous lemmae, given a GCDA $A$, we have an effective method to determine if $A$ is $i$-formal or not. It consists on the following:
\begin{itemize}
 \item Compute a presentation of $M_A$, the $i$-minimal model of $A$. During the construction, we get the numerical invariants of $M$.
 \item Compute a presentation of $H^*(M_A)$ up to degree $i+1$.
 \item Compute a presentation of $M_H$, the $i$-minimal model of $H^*(M_A)$. During the construction, we get the numerical invariants of $M_H$.
 \item If the numerical invariants of $M_A$ and $M_H$ do not coincide, $A$ is not $i$-formal.
\item If the numerical invariants do coincide, check if the presentation of $M_A$ satisfies the $\psi$-condition. If it does, $A$ is $i$-formal.
\end{itemize}

In every case that we have tested, this criterion has been able to determine if
a GCDA is $i$-formal or not. This fact motivates the following questions:

\begin{question}
 Is the equality of numerical invariants up to degree $i$ a sufficient condition for $i$-formality?
\end{question}

\begin{question}
 Is the $\psi$-condition up to degree $i$ necessary for $i$-formality?
\end{question}

\section{Implementation and examples}


We present an implementation of the previous algorithms in SageMath (\cite{sage}). The computation of the minimal models and cohomology algebras have been already included since version 8.8. The formality criterion is currently under development but the medium term goal is to include it in future SageMath versions. The current code is tracked by the following ticket: \texttt{https://trac.sagemath.org/ticket/28155}.

We now illustrate it with an example.

\subsection{Examples of computations of minimal models with SageMath}

\begin{exmp}

    In SageMath we can define the differential algebra in Example~\ref{exm:algorithm} as follows:

\begin{verbatim}
sage: A.<e1,e2,e3,e4,e5,e6,e7> = GradedCommutativeAlgebra(QQ,
degrees=[1,1,1,1,1,1,2])
sage: B = A.cdg_algebra({e1:-e1*e6,e2:-e2*e6,e3:-e3*e6,e4:-e5*e6})
\end{verbatim}
Its $4$-minimal model can be computed by
\begin{verbatim}
sage: phi = B.minimal_model(4)
sage: phi
Commutative Differential Graded Algebra morphism:
  From: Commutative Differential Graded Algebra with generators (
  'x1_0', 'x1_1', 'y1_0', 'x2_0') in degrees (1, 1, 1, 2) over
  Rational Field with
differential:
   x1_0 --> 0
   x1_1 --> 0
   y1_0 --> x1_0*x1_1
   x2_0 --> 0
  To:   Commutative Differential Graded Algebra with generators (
  'e1', 'e2', 'e3', 'e4', 'e5', 'e6', 'e7') in degrees (1, 1, 1, 1,
  1, 1, 2) over Rational Field with differential:
   e1 --> -e1*e6
   e2 --> -e2*e6
   e3 --> -e3*e6
   e4 --> -e5*e6
   e5 --> 0
   e6 --> 0
   e7 --> 0
  Defn: (x1_0, x1_1, y1_0, x2_0) --> (e6, e5, e4, e7)
\end{verbatim}

Notice that the result is given as a differential algebra morphism from the $i$-minimal model to the algebra given as input. That is, we get not only an abstract description of the $i$-minimal model, but also an explicit $i$-quasi-isomorphism. We can get the $i$-minimal model itself as the domain of the morphism:

\begin{verbatim}
sage: phi.domain()
Commutative Differential Graded Algebra with generators ('x1_0',
'x1_1', 'y1_0', 'x2_0') in degrees (1, 1, 1, 2) over Rational Field
with differential:
   x1_0 --> 0
   x1_1 --> 0
   y1_0 --> x1_0*x1_1
   x2_0 --> 0
\end{verbatim}

Note that here we see right part of the diagram~\eqref{diagram:ejemplo}, whereas in the line
\begin{verbatim}
  Defn: (x1_0, x1_1, y1_0, x2_0) --> (e6, e5, e4, e7)
\end{verbatim}
we see the left part.

%
%
%

\end{exmp}

\begin{exmp}
We can also work with non-free algebras. They must be introduced as the quotient
of a free algebra by a bilateral ideal. For instance, the cohomology algebra of
$\mathbb{S}^2\vee\mathbb{S}^3$ has only elements in degrees $2$ and $3$.

\begin{verbatim}
sage: A.<e2,e3> = GradedCommutativeAlgebra(QQ, degrees=[2,3])
sage: I = A.ideal([e2^2, e2*e3])
sage: Q = A.quotient(I)
sage: Q
Graded Commutative Algebra with generators ('e2', 'e3') in degrees
(2, 3) with relations [e2^2, e2*e3] over Rational Field
\end{verbatim}
We can check that this algebra only has elements in degrees $2$ and $3$:
\begin{verbatim}
sage: Q.basis(2)
[e2]
sage: Q.basis(3)
[e3]
sage: Q.basis(4)
[]
sage: Q.basis(5)
[]
\end{verbatim}
Now we define its corresponding GCDA with trivial differential, and compute its
$6$-minimal model.
\begin{verbatim}
sage: B = Q.cdg_algebra({})
sage: B.minimal_model(6)
Commutative Differential Graded Algebra morphism:
  From: Commutative Differential Graded Algebra with generators (
  'x2_0', 'x3_0', 'y3_0', 'y4_0', 'y5_0', 'y6_0', 'y6_1') in degrees
  (2, 3, 3, 4, 5, 6, 6) over Rational Field with differential:
   x2_0 --> 0
   x3_0 --> 0
   y3_0 --> x2_0^2
   y4_0 --> x2_0*x3_0
   y5_0 --> x3_0*y3_0 + x2_0*y4_0
   y6_0 --> -y3_0*y4_0 + x2_0*y5_0
   y6_1 --> x3_0*y4_0
  To:   Commutative Differential Graded Algebra with generators (
  'e2', 'e3') in degrees (2, 3) with relations [e2^2, e2*e3] over
  Rational Field with differential:
   e2 --> 0
   e3 --> 0
  Defn: (x2_0, x3_0, y3_0, y4_0, y5_0, y6_0, y6_1) --> (e2, e3, 0, 0, 0, 0, 0)

\end{verbatim}

\end{exmp}

\subsection{Examples of formality criteria with SageMath}

In \cite{bock}, Bock studied the formality of solvmanifolds up to dimension $6$.
In the following we show some examples that were not covered there.

\begin{exmp}
The algebra $G^0_{5.14}$ in \cite{bock} is not formal (in fact, not even $2$-formal):
\begin{verbatim}
sage: A.<x1,x2,x3,x4,x5> = GradedCommutativeAlgebra(QQ)
sage: B = A.cdg_algebra({x1:-x2*x5,x4:x3*x5,x3:-x4*x5})
sage: B.is_formal(2)
False
\end{verbatim}

Indeed, we can look at the $3$-minimal model:

\begin{verbatim}
sage: B.minimal_model(3).domain()
Commutative Differential Graded Algebra with generators (
'x1_0', 'x1_1', 'y1_0', 'x2_0', 'y3_0') in degrees (1, 1,
1, 2, 3) over Rational Field with differential:
   x1_0 --> 0
   x1_1 --> 0
   y1_0 --> x1_0*x1_1
   x2_0 --> 0
   y3_0 --> x2_0^2
\end{verbatim}
We can see that the nonzero numerical invariants are $v^1_1=1$,  $v^2_0=1$ and $v^3_1=1$.

If we try to compute the $2$-minimal model of the cohomology algebra, we get:

\begin{verbatim}
sage: H = B.cohomology_algebra(3)
sage: H
Commutative Differential Graded Algebra with generators (
'x0', 'x1', 'x2', 'x3', 'x4') in degrees (1, 1, 2, 2, 2)
with relations [x0*x1, x0*x2, x1*x2 + x0*x4, x1*x4] over
Rational Field with differential:
   x0 --> 0
   x1 --> 0
   x2 --> 0
   x3 --> 0
   x4 --> 0
sage: H.minimal_model(2)
...
ValueError: could not cover all relations in max iterations in degree
2
\end{verbatim}

This means that the algorithm did not finish after $3$ iterations of step~\ref{step:y} in degree $1$ ($3$ is the default value to decide to give up). This implies that $M_H$ has more than $3$ nonzero numerical invariants $v^1_1,v^1_2,v^1_3$ (in fact, it can be checked that for $M_H$, the first numerical
invariants are $v^1_0=2$ ,$v^1_1=1$,$ v^1_2=2$ and $v^1_3=3$), and hence it cannot be isomorphic to $M_A$.
\end{exmp}

\begin{exmp}
    The case $G^{-2,0}_{5.35}$ in \cite{bock} is $6$-formal:
    \begin{verbatim}
sage: A.<x1,x2,x3,x4,x5> = GradedCommutativeAlgebra(QQ)
sage: B = A.cdg_algebra({x1:2*x1*x4,x2:-x2*x4-x3*x5,x3:-x3*x4+x2*x5})
sage: B.is_formal(6)
True
    \end{verbatim}

We can actually see that it is indeed formal. Since the algebra is generated by $5$ generators of degree $1$, it is trivial beyond degree $5$. We can see that its $5$-minimal model is also trivial beyond degree $5$ (moreover, it coincides with its cohomology algebra):

\begin{verbatim}
sage: B.minimal_model(5)
Commutative Differential Graded Algebra morphism:
  From: Commutative Differential Graded Algebra with generators (
  'x1_0', 'x1_1', 'x3_0') in degrees (1, 1, 3) over Rational Field
  with differential:
   x1_0 --> 0
   x1_1 --> 0
   x3_0 --> 0
  To:   Commutative Differential Graded Algebra with generators (
  'x1', 'x2', 'x3', 'x4', 'x5') in degrees (1, 1, 1, 1, 1) over
  Rational Field with differential:
   x1 --> 2*x1*x4
   x2 --> -x2*x4 - x3*x5
   x3 --> -x3*x4 + x2*x5
   x4 --> 0
   x5 --> 0
  Defn: (x1_0, x1_1, x3_0) --> (x4, x5, x1*x2*x3)
\end{verbatim}
So the $5$-minimal model is in fact the minimal model. And it is trivially isomorphic to the cohomology algebra.

\end{exmp}

\end{document}